\documentclass{amsart}
\usepackage{amsmath,amsthm}
\usepackage{mathtools}
\usepackage{amsfonts,amssymb}
\usepackage{accents}
\usepackage{enumerate}
\usepackage{accents,color}
\usepackage{graphics,graphicx,tikz}

\newcommand{\lvt}{\left|\kern-1.35pt\left|\kern-1.3pt\left|}
\newcommand{\rvt}{\right|\kern-1.3pt\right|\kern-1.35pt\right|}
\makeatletter
\@namedef{subjclassname@2020}{\textup{2020} Mathematics Subject Classification}
\makeatother

\newtheorem{thm}{Theorem}[section]
\newtheorem{cor}[thm]{Corollary}
\newtheorem{lem}[thm]{Lemma}
\newtheorem{prop}[thm]{Proposition}
\newtheorem{exam}[thm]{Example}

\newtheorem{defn}[thm]{Definition}
\newtheorem{assumption}[thm]{Assumption}

\theoremstyle{remark}
\newtheorem{rem}{Remark}[section]

\numberwithin{equation}{section}

 \def\x{{\mathbf{x}}}
 \def\y{{\mathbf{y}}}

 \def\v{{\mathbf{v}}}
 \def\u{{\mathbf{u}}}
 \def\w{{\mathbf{w}}}

 \def\a{{\alpha}}

 \def\NN{{\mathbb N}}

 \def\RR{{\mathbb R}}

\def\lla{\langle{\kern-2.5pt}\langle}
\def\rra{\rangle{\kern-2.5pt}\rangle}

\DeclareMathOperator{\sgn}{sgn}

\graphicspath{{./}}

\newcommand{\rev}[1]{\textcolor{black}{#1}}
\begin{document}

\title[Least Chebyshev polynomials]{Least multivariate Chebyshev polynomials on diagonally-determined sets}

\author{M. Dressler}
\address{Mareike Dressler, School of Mathematics and Statistics, University of New South Wales, Sydney, NSW 2052, Australia.}
\email{m.dressler@unsw.edu.au}

\author{S. Foucart}
\address{Simon Foucart, Department of Mathematics, Texas A\&M University, United States.}
\email{foucart@tamu.edu}

\author{M.  Joldes}
\address{Mioara  Joldes, LAAS-CNRS, France.}
\email{joldes@laas.fr}

\author{E. de Klerk}
\address{Etienne de Klerk, Department of Econometrics and Operations Research, Tilburg University, The Netherlands.}
\email{E.deKlerk@tilburguniversity.edu}

\author{J.~B. Lasserre}
\address{Jean-Bernard Lasserre, LAAS-CNRS and Toulouse School of Economics, France.}
\email{lasserre@laas.fr}

\author{Y. Xu}
\address{Yuan Xu, Department of Mathematics, University of Oregon, United States.}
\email{yuan@uoregon.edu}


\begin{abstract}
We consider a new multivariate generalization of the classical monic (univariate) Chebyshev polynomial that minimizes the uniform norm
on the interval $[-1,1]$. Let $\Pi^*_n$ be the subset of polynomials of degree at most $n$ in $d$ variables, whose homogeneous part of degree $n$ has  coefficients summing up to $1$. The problem is determining a polynomial in  $\Pi^*_n$
with the smallest uniform norm on a set $\Omega \subset \mathbb{R}^d$, which we call a least Chebyshev polynomial (associated with $\Omega$). Our main result
solves the problem for $\Omega$ belonging to a non-trivial class of sets that we call {\em diagonally-determined}, and establishes the remarkable result that a least Chebyshev 
polynomial can be given via the classical, univariate, Chebyshev polynomial. In particular, the solution
can be independent of the dimension. Diagonally-determined sets include centered balls in $\mathbb{R}^d$ in any norm, but can be non-convex and
even non-simply connected. We also introduce a computational procedure, based on semidefinite programming hierarchies, to detect if a given semi-algebraic set is diagonally-determined.
\end{abstract}

\maketitle

\section{Introduction}

Among its numerous properties, the Chebyshev polynomial $T_n(x) =\cos (n \arccos x)$ provides a solution
for the best approximation to the monomial $x^n$ on the interval $[-1,1]$ in the uniform norm. More precisely,
the polynomial
$$
   q_{n}^*(x) = x^n - 2^{1-n} T_n(x)
$$
of degree $n-1$ ($n \in \mathbb{N}$) is the best polynomial of approximation to $x^n$ on $[-1,1]$; that is
\begin{equation} \label{intro-def-uni}
    q_n^*\,=\,\arg\min_{q\in\Pi_{n-1}}\,\sup_{x\in [-1,1]}\vert x^n-q(x)\vert\,,
\end{equation}
with $\Pi_{n-1}$ being the vector space of univariate polynomials of degree at most $n-1$.
In other words, the monic Chebyshev polynomial $x^n-q^*_n$ is the {\it least polynomial} in the sense that it
has the least uniform norm among all monic polynomials of degree $n$.

There have been multiple extensions of Chebyshev polynomials to multivariate settings from different angles. From the point of view of approximation, an immediate generalization is finding the best approximation to monomials
\cite{AV1,AV2,G,BHN, MoaPeh, NX,R,Sloss,X04,X05}.
To this end we replace the interval $[-1,1]$ by a general set $\Omega \subset \mathbb{R}^d$, for $d>1$. Letting $\a\in\NN^d$ and $\vert\a\vert \coloneqq \sum_{i=1}^d \alpha_i =n$, we consider the problem
\begin{equation} \label{def-intro-multi-1}
\min_{q\in\Pi^d_{n-1}}\,\sup_{\x\in \Omega}\,\vert \x^{\a}-q(\x)\vert\,,
\end{equation}
where $\Pi^d_{n-1}$ denotes the space of real polynomials of total degree at most $n-1$ in the variables $\x=(x_1,\ldots,x_d)$, and where we define the monomial $\x^\alpha \coloneqq x_1^{\alpha_1}\ldots x_d^{\alpha_d}$. Since we will deal with monomials of degree $n$ throughout the paper, we introduce the notation
\[
\NN^d_n \coloneqq \left\{ \alpha \in (\NN_0)^d \; : \; |\alpha | = n\right\},
\]
where $\NN_0 = \NN \cup \{0\}$ denotes the set of nonnegative integers. Thus $\NN^d_n$ indexes all monomials of degree {\em exactly} $n$ in $d$ variables.

Problem \eqref{def-intro-multi-1} can be regarded as a natural multivariate generalization of \eqref{intro-def-uni}, where $\Omega\subset \RR^d$
is a subset of $\RR^d$. While the interval $[-1,1]$ is a prototype of a compact connected set of the real line,
 there is no `prototype' in higher dimensions for such a set. In the literature, this problem has been
studied primarily on a few special regular domains. In two variables, the problem \eqref{def-intro-multi-1} is solved
for the square, the disk, and the isosceles right triangle \cite{G, BHN, NX,R, Sloss}.
While the solution on the square can be extended to the cube for $d > 2$, the problem is solved only for a few cases,
mostly monomials of lower degrees, on the ball and the simplex \cite{AV1, AV2, X04}.
 Moreover, the existing examples indicate an increasing complexity, so much so that it does
not appear possible to find an analytic solution even for these regular domains.

Recently, in \cite{caltex-1} we have proposed to investigate \eqref{def-intro-multi-1} for various choices of $\Omega\subset\RR^d$
and $\a\in \NN^d_n$ by combining analytical tools with numerical tools from optimization (and notably the so-called
moment-SOS hierarchy \cite{lass-siopt-01}). During this study, we have encountered an optimization problem that has initiated a change of view: Namely, instead of studying the best polynomial of approximation to monomials, we can study the least polynomial instead.
While the two concepts are identical in one variable, they can be quite different in higher dimensions,
as seen from the definition below.

We emphasize that -- even though our paper deals with the characterization of least polynomials on special sets -- it is intrinsically connected to constructive approximation and numerical optimization due to the close links with the work in \cite{caltex-1} and \cite{lass-siopt-01} respectively.

\begin{defn}
\label{def:extremal_problem}
Let $\Pi_n^d$ denote the space of polynomials of total degree at most $n$ in $d$ variables,
and $\Pi^*_n$ the subset of $\Pi_n^d$ that consists of polynomials of the form 
\[
\x\mapsto P(\x)\,\coloneqq\,\sum_{\a \in \NN^d_n}a_\a\,\x^\a+Q(\x) \, \text{ with }\quad \sum_{\a \in \NN^d_n}a_\a=1\,\text{and }\:Q\in\Pi^d_{n-1}.
\]

Let $\Omega$ be a set in $\RR^d$. We consider the optimization problem
\begin{equation}
    \label{def-intro-multi-2}
\inf_{P\in\Pi^*_n}\, \|P\|_\Omega, \quad \hbox{where} \quad \|P\|_\Omega \coloneqq \sup_{\x\in \Omega}\,\vert P(\x)\vert\,.
\end{equation}
If it exists, we call a minimizer $P^*\in\Pi^*_n$ of \eqref{def-intro-multi-2} a \emph{least Chebyshev polynomial} of degree $n$ on the set $\Omega$.
\end{defn}

For $d =1$, there is only one monomial of degree $n$. In the case $d > 1$, every element of $\Pi^*_n$ is
`monic' and the monomial $\x^\a$ in \eqref{def-intro-multi-1} is only one among many possible choices in $\Pi^*_n$.
However, rather than approximating a fixed monomial  by polynomials of lower degree, the problem
\eqref{def-intro-multi-2} requires finding a polynomial that has the least norm among all polynomials in $\Pi_n^*$.
As far as we are aware, this problem has not been considered in the literature.

\subsection*{Contribution}
The main purpose of this paper is to report our findings on the optimization problem~\eqref{def-intro-multi-2}. It turns out,
much to our surprise, that the problem~\eqref{def-intro-multi-2} can be solved analytically for a fairly general family of sets
$\Omega$ in $\RR^d$ for all $d \ge 2$.
This family of sets will be referred to as {\em diagonally-determined}.

\subsection*{Organization of this paper}
The paper is organized as follows. We start by defining and describing several examples of diagonally-determined
 sets in Section \ref{sec:examples}. Our main results are presented in Section~\ref{subsec:main-diagonal-supported},
  where we describe the least polynomial for a diagonally-determined set; see Theorem \ref{thm:supported}. 
  We discuss  the dual problem of \eqref{def-intro-multi-2} in Section~\ref{sec:dual-framework-signatures} and show that it too has a
  closed-form solution in the case of diagonally-determined sets; see Theorem \ref{thm:sinature_diag_supp}. We also rephrase this result for the dual problem in the framework of {\em extremal signatures}.
\rev{Finally, we consider the problem of deciding whether a given semi-algebraic set is diagonally determined in Section \ref{sec:detection}, and propose an algorithmic approach based on semidefinite programming hierarchies.
We conclude the paper with a discussion of related problems and open questions in Section \ref{sec:conclusions}.}

\section{Diagonally-determined sets}
\label{sec:examples}
We define the {\em diagonal} of a set $\Omega \subset \mathbb{R}^d$ as the set
${\rm  diag}(\Omega)\coloneqq \{ t \in \RR: t \mathbf{1} \in \Omega \}$, where $\mathbf{1}$ denotes the all-ones vector in $\mathbb{R}^d$.
For two vectors $\x,\y\in\RR^d$, we use $\langle \x,\y\rangle$ for their standard inner product.

\begin{defn}
\label{def:diag_concentration}
We call a set $\Omega \subset \mathbb{R}^d$ {\em diagonally-determined} if the following two conditions hold:
\begin{enumerate}
\item The diagonal of $\Omega$ is an interval, say ${\rm  diag}(\Omega) = [a,b]$, and
\item there exists $\v \in \RR^d$ such that $\langle \v, \mathbf{1} \rangle = 1$ and $\langle \v,\x \rangle \in [a,b]$ for all $\x \in \Omega$.
\end{enumerate}
\end{defn}

Importantly, Definition \ref{def:diag_concentration} covers sets that could be non-convex, non-compact, and not simply connected.
Figure~\ref{fig-cheby-1}
gives an  example of a non-convex  diagonally-determined set in $\RR^2$.

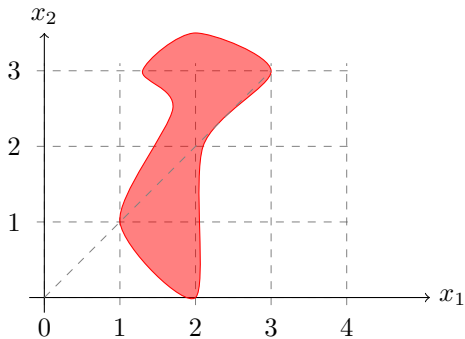
\begin{figure}[ht] \label{fig:1}
\begin{center}
\begin{tikzpicture}[domain=0:5]
\draw[very thin,color=gray,dashed] (-0.1,-0.1) grid (4.1,3.1);
\draw[->] (-0.2,0) -- (5.1,0) node[right] {$x_1$};
\draw[->] (0,-0.2) -- (0,3.5) node[above] {$x_2$};

\filldraw [red, fill opacity=0.5] plot [smooth cycle] coordinates { (1,1) (1.7,2.5) (1.3,3) (2,3.5) (3,3) (2.1,2) (2,0)};

\draw[very thin,color=gray,dashed] (0,0) -- (3,3);

\draw (0,-0.4) node {$0$};
\draw (1,-0.4) node {$1$};
\draw (2,-0.4) node {$2$};
\draw (3,-0.4) node {$3$};
\draw (4,-0.4) node {$4$};
\draw (-0.4,1) node {$1$};
\draw (-0.4,2) node {$2$};
\draw (-0.4,3) node {$3$};

\end{tikzpicture}
\caption{Example of a diagonally-determined, non-convex set. Here, $\mbox{diag}(\Omega) = [1,3]$, and $\v = (1,0)$. \label{fig-cheby-1}}
\end{center}
\end{figure}

As a first observation, a dilation of a diagonally-determined set is again diagonally-determined.

\begin{lem}
\label{prop:dilation}
If $\Omega \subset \mathbb{R}^d$ is diagonally-determined with ${\rm  diag}(\Omega) = [a,b]$ and vector $\v$, then, for any $r >0$,
the dilation
\[
r \Omega \coloneqq \{r\x \; : \; \x\in \Omega\}
\]
is also diagonally-determined with ${\rm  diag}(r\Omega) = [ra,rb]$ and vector $\v$.
\end{lem}
\begin{proof}
The proof is an immediate consequence of Definition \ref{def:diag_concentration}.
\end{proof}

We will now show that  balls in $\RR^d$ (in any norm), \rev{centered at the origin}, are examples of diagonally-determined sets.
To this end, we first need to recall some definitions.

\begin{defn}
For a given norm $\|\cdot\|$ on $\mathbb{R}^d$,
we use $\|\cdot\|_*$ to denote its dual norm, namely
\[
\|\u\|_* = \sup \{ \langle \u,\x\rangle \; : \; \|\x\| = 1\}, \quad \u \in \RR^d.
\]
\end{defn}
 We  will also consider the special case of monotone norms.
\begin{defn}
We call a norm $\|\cdot\|$ on $\mathbb{R}^d$ {\em monotone}, if
\[
\| \x\| \le \| |\x| \| \quad \forall \x \in \mathbb{R}^d,
\]
where $|\x| = (|x_1|,\ldots,|x_d|)$.
\end{defn}
%

We may now show that any centered ball is diagonally-determined.
\begin{prop}
\label{prop:ball}
Consider a  ball of radius $r$ in $\RR^d$ centered at the origin:
\[
 \Omega \coloneqq \{ \x \in \RR^d \; : \; \|\x\| \le r\},
\]
where $\|\cdot\|$ denotes any norm on $\RR^d$.
Then $\Omega$ is a diagonally-determined set. If, in addition, the norm $\|\cdot\|$ is monotone,then the vector $\v$ may be assumed to be entrywise-nonnegative, denoted by \ $\v \ge \mathbf{0}$.
\end{prop}
\begin{proof}
We prove the statement for unit balls; the required result then follows from Lemma \ref{prop:dilation}.
For a centered unit ball, one has ${\rm diag}(\Omega) = [a,b]$ with \rev{$a = -1/\|\mathbf{1}\|$ and} $b = 1/\|\mathbf{1}\|$.

Now let $\v = \u / \|\mathbf{1}\|$, where $\u \in \RR^d$ is `dual' to $\mathbf{1}$, in the sense that $\|\u\|_* = 1$
 and $\langle \u, \mathbf{1} \rangle = \|\mathbf{1}\|$.
We then have $\langle \v , \mathbf{1} \rangle = \langle \u / \|\mathbf{1}\|, \mathbf{1} \rangle = 1$, as required.
Moreover we have, for any $\x \in \Omega$,
\[
|\langle \v, \x \rangle| = |\langle \u / \|\mathbf{1}\|, \x \rangle|  \le  (1/\|\mathbf{1}\|) \|\u\|_* \|\x\| \le 1/\|\mathbf{1}\|.
\]
Thus we see that $\langle \v, \x \rangle \in  [a,b]$ for all $\x \in \Omega$, proving that $\Omega$ is diagonally-determined.

Assume now that the norm $\|\cdot\|$ is monotone.
We have
\[
\langle \u  , \mathbf{1} \rangle = \|\mathbf{1}\| \ge \|\mbox{sgn}(\u)\| \ge \langle \u  , \mbox{sgn}(\u) \rangle,
\]
where, for $i \in \{1,\ldots,d\}$,
\[
\mbox{sgn}(\u)_i = \left\{
\begin{array}{rl}
1 & \mbox{if $u_i \ge 0$} \\
-1 & \mbox{else.}
\end{array}
\right.
\]
This means that $\sum_{i=1}^d u_i \ge \sum_{i=1}^d |u_i|$, which holds if and only of $\u \ge \mathbf{0}$.
Subsequently, since $\v = \u / \|\mathbf{1}\|$, we have $\v \ge \mathbf{0}$.

\end{proof}

By Definition \ref{def:diag_concentration}, we have the following immediate corollary.
\begin{cor}
Assume $\Omega \subset \RR^d$ is a subset of a  ball in $\RR^d$ in any norm, centered at the origin,  and that the diagonal of $\Omega$ coincides with the diagonal of the ball.
Then $\Omega$ is  a diagonally-determined set.
\end{cor}

To illustrate this corollary, the  example in Figure \ref{fig-cheby-3} shows a non-convex subset of a unit (Euclidean) ball, that has the same diagonal as the ball.
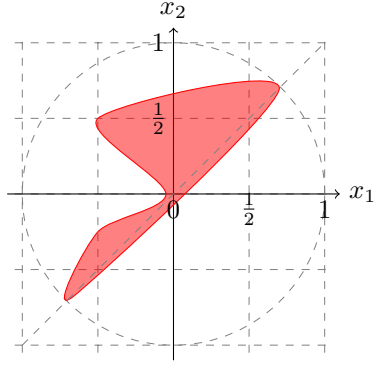
\begin{figure}[ht]
\begin{center}
\begin{tikzpicture}[domain=-2:2]
\draw[very thin,color=gray,dashed] (-2.1,-2.1) grid (2.1,2.1);
\draw[->] (-2.2,0) -- (2.2,0) node[right] {$x_1$};
\draw[->] (0,-2.2) -- (0,2.2) node[above] {$x_2$};

\draw[very thin,color=gray,dashed] (0,0) circle [radius=2];

\filldraw [red, fill opacity=0.5] plot [smooth cycle] coordinates { (-1.33,-1.36) (-1,-0.5) (-0.1,0) (-1,1)  (1.4,1.39) };

\draw[very thin,color=gray,dashed] (-2,-2) -- (2,2);

\draw (0,-0.2) node {$0$};
\draw (1,-0.2) node {$\frac{1}{2}$};
\draw (2,-0.2) node {$1$};
\draw (-0.2,1) node {$\frac{1}{2}$};
\draw (-0.2,2) node {$1$};

\end{tikzpicture}
\caption{Example of a diagonally-determined, non-convex set. Here,  $\mbox{diag}(\Omega) = \left[-\frac{1}{\sqrt{2}},\frac{1}{\sqrt{2}}\right]$,
 and $\v = \left(\frac{1}{2},\frac{1}{2}\right)$. \label{fig-cheby-3}}
\end{center}
\end{figure}

\begin{exam}
\label{ex:wol}
It is insightful to consider the specific example of the unit ball of the {\sf owl} norm (ordered weighted $\ell_1$-norm) defined relative to weights $w_1 \ge w_2 \ge \cdots \ge w_d \ge 0$ by
$$
\|\x\|_{\sf owl} = \sum_{i=1}^d w_i x^*_i,
$$
with
$(x_1^*,\ldots,x_d^*)$ being the nonincreasing rearrangement of $(|x_1|,\ldots,|x_d|)$.
Note that $\|\mathbf{1}\|_{\sf owl} = \sum_{i=1}^d w_i \eqqcolon W$, so that $b = 1/W$.
Letting $\w = (w_1,\ldots,w_d)$, we set $\v = \w/W$, since
$\langle \w/W, \mathbf{1} \rangle = (\sum_{i=1}^d w_i)/W = 1$
and, for $\|\x\|_{\rm owl} \le 1$,
\begin{eqnarray*}
  |\langle \w/W, \x \rangle | &=& \left|\sum_{i=1}^n w_i x_i \right| /W \\
   &\le & \left(\sum_{i=1}^n w_i | x_i |\right) /W \\
   &\le & \left(\sum_{i=1}^n w_i x^*_i \right) /W \\
   &\le & 1/W = b.
\end{eqnarray*}
\end{exam}

The intersection of certain balls in $\RR^d$ with the nonnegative orthant in $\RR^d$ is also diagonally-determined, as the next result shows.
\begin{prop}
\label{prop:ell_p_ex}
Assume $\|\cdot\|$ is a norm on $\RR^d$ with dual norm $\|\cdot \|_*$, and that the following holds:
there is an entrywise-nonnegative vector  $\u \ge \mathbf{0}$ such that $\|\u \|_*= 1$ and $\langle \u, \mathbf{1}\rangle = \|\mathbf{1}\|$.

Further assume  $\Omega \subset \RR^d$ has diagonal $\mathrm{diag}(\Omega) = [0,r/\|\mathbf{1}\|]$ for some $r>0$, and,  for all $\x \in \Omega$,
 $\x \ge \mathbf{0}$  and $\|\x\| \le r$.
Then $\Omega$ is a diagonally-determined set with vector $\v = \mathbf{u}/\|\mathbf{1}\|$.
\end{prop}
\begin{proof}
We again prove the statement for the case $r=1$; the result for general $r>0$ then follows from Lemma \ref{prop:dilation}.
Setting $\v = \mathbf{u}/\|\mathbf{1}\|$, one has $\langle \v, \mathbf{1} \rangle =1 $.
Moreover, for $\x \in \Omega$,
 $\langle \v, \x \rangle \ge 0$,
since $\x,\v \ge 0$, and
\[
\langle \v, \x \rangle = \frac{1}{\|\mathbf{1}\|}\langle \mathbf{u}, \x \rangle \le \frac{1}{\|\mathbf{1}\|}\|\u\|_*\|\x\| \le \frac{1}{\|\mathbf{1}\|},
\]
so that $\langle \v, \x \rangle \in \mathrm{diag}(\Omega)$ for all $\x \in \Omega$.
\end{proof}
We note again that the assumption on the  norm in Proposition \ref{prop:ell_p_ex} is met by, for example, all
{\em monotone norms}, as shown in the proof of Proposition \ref{prop:ball}.

As a simple corollary of the proposition, the simplex
\[
\Omega = \left\{ \x \in \RR^d \; : \; \x \ge \mathbf{0}, \;  \sum_{i=1}^d x_i\le 1\right\},
\]
is diagonally-determined, since  it is the intersection of the unit $\ell^1$-ball with the nonnegative orthant.
See Figure~\ref{fig-cheby-2} for an illustrative example of Proposition \ref{prop:ell_p_ex}.

\begin{figure}[ht]
\begin{center}
\begin{tikzpicture}[domain=0:4]
\draw[very thin,color=gray,dashed] (-0.1,-0.1) grid (3.1,3.1);
\draw[->] (-0.2,0) -- (3.6,0) node[right] {$x_1$};
\draw[->] (0,-0.2) -- (0,3.5) node[above] {$x_2$};

\filldraw [red, fill opacity=0.5] plot [smooth cycle] coordinates { (0,0.015) (1,2) (2,2.25) (2.7,2.665) (2,1) (1.1,0.95) };

\draw[very thin,color=gray,dashed] (0,0) -- (3,3);

\draw (0,-0.4) node {$0$};
\draw (1,-0.4) node {$\frac{1}{2}$};
\draw (2,-0.4) node {$1$};
\draw (3,-0.4) node {$\frac{3}{2}$};
\draw (-0.4,1) node {$\frac{1}{2}$};
\draw (-0.4,2) node {$1$};
\draw (-0.4,3) node {$\frac{3}{2}$};

\end{tikzpicture}
\caption{Example to illustrate Proposition \ref{prop:ell_p_ex}. Here, $\mbox{diag}(\Omega) = [0,b]$, with $b \approx 1.35$. Moreover, $\Omega$ is
contained in the $\ell^1$-ball of radius $r = b{d} \approx 2.7$, intersected with the nonnegative quadrant. \label{fig-cheby-2}}
\end{center}
\end{figure}
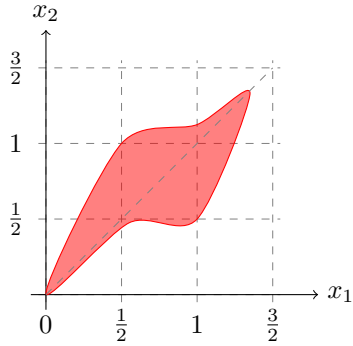

\rev{
Finally, the property of being diagonally-determined is not invariant under translation in general. The exception is a translation by a constant multiple of the all-ones vector.
We state this observation as a lemma for later use. The proof again follows immediately from Definition \ref{def:diag_concentration}
and is omitted.
\begin{lem}
\label{lemma:translation}
Fix $\alpha \in \mathbb{R}$. Then a given $\Omega \subset \mathbb{R}^d$  is diagonally determined with $\mathrm{diag}(\Omega) = [a,b]$ and vector $\v$ if and only if its translation $\Omega + \alpha \mathbf{1}$ is diagonally determined
with diagonal $[a + \alpha, b + \alpha]$ and vector $\v$.
\end{lem}
}

\section{Least Chebyshev polynomial for a diagonally-determined set}
\label{subsec:main-diagonal-supported}

In this section, and with $\Omega$ being a set in $\RR^d$, we consider the problem of finding a least polynomial $P^* \in \Pi_n^*$ such that
$$
     \|P^*\|_\Omega =  \inf \left  \{ \|P\|_{\Omega}: P \in \Pi_n^* \right \},
$$
for the special case where $\Omega$ is diagonally-determined.

\begin{thm} \label{thm:supported}
Let $\Omega$ be diagonally-determined with vector $\v$ and $\mathrm{diag}(\Omega) = [a,b]$. Then
\begin{equation}\label{eq:inf}
  \inf \left  \{ \|P\|_{\Omega}: P \in \Pi_n^* \right \} = \left(\frac{b-a}{2}\right)^n \frac{1}{2^{n-1}}\,,
\end{equation}
and  the infimum is attained by the polynomial
\begin{equation}
    \label{eq:minimizer}
   \x\mapsto P^*(\x) \,\coloneqq\, \left(\frac{b-a}{2}\right)^n \frac{1}{2^{n-1}}   T_n\left(- 1+ 2 \,\frac{\langle \v,\x\rangle -a}{b-a}\right).
\end{equation}
\end{thm}

\begin{proof}
Let $P \in \Pi_n^*$, say
\[
P(\x) = \sum_{\a \in \NN^d_n}a_\a\,\x^\a+Q(\x) \, \quad \text{with }\quad \sum_{\a \in \NN^d_n}a_\a=1\,\text{and }\:Q\in \Pi^d_{n-1}.
\]
 Then $P(t,\ldots, t) = t^n + q_{n-1}(t)$, where $q_{n-1}(t) = Q(t,\ldots,t)$ is a univariate polynomial of degree at most $n-1$
in the single variable $t$. Hence,
\begin{equation}\label{eq:lwbd}
   \|P\|_\Omega  \ge \max_{a \le t \le b} \left | t^n + q_{n-1}(t) \right| \ge  \left(\frac{b-a}{2}\right)^n \frac{1}{2^{n-1}}\,,
\end{equation}
where we have used a classical result in one variable (e.g.\ \cite[Theorem 2.1]{Riv}), and equality in the last inequality is
 attained by choosing $q_{n-1}$
such that $t^n + q_{n-1}(t)$ is the rescaled Chebyshev polynomial in the right-hand side of \eqref{eq:minimizer}, when restricted to $\x = (t,...,t)$ ($t \in \RR$).

Moreover, as $\Omega$ is diagonally-determined, then for all $\x \in \Omega$, one has $\langle \v,\x \rangle \in [a,b]$, and therefore
$$
  \|P^*\|_\Omega = \left(\frac{b-a}{2}\right)^n \frac{1}{2^{n-1}}  \max_{\x \in \Omega} \left| T_n\left(- 1+ 2 \frac{\langle \v,\x \rangle  -a}{b-a}\right) \right|
   \le \left(\frac{b-a}{2}\right)^n \frac{1}{2^{n-1}}\,.
$$
Finally, we need to show that $P^* \in \Pi_n^*$. First note that $t \mapsto P^*(t \mathbf{1})$ is a
 monic univariate polynomial in $t$, since $\langle \v,\mathbf{1}\rangle = 1$.
On the other hand, we may write $P^*$ in the form
\[
P^*(\x) = \sum_{\alpha \in \NN^d_n}a^*_\a\,\x^\a+Q^*(\x) \,  \text{ with } Q^*\in \Pi^d_{n-1},
\]
so that
\[
P^*(t \mathbf{1}) = \sum_{\a \in \NN^d_n}a^*_\a t^n +
Q^*(t \mathbf{1}).
\]
Thus $\sum_{\a\in \NN^d_n}a^*_\a=1$, and therefore $P^* \in \Pi_n^*$, as required.
\end{proof}

The minimal value in \eqref{eq:inf} does not \rev{depend} explicitly on the dimension $d$, but the diagonal $[a,b]$ may. For example, if $\Omega$ is the unit Euclidean ball in $\RR^d$, then \rev{ $[a,b] = [-1/\sqrt{d},1/\sqrt{d}]$}.

\begin{exam}
For the cube $\Omega=[-1,1]^d$, Theorem \ref{thm:supported} shows that $\|P^*\|_\Omega = 2^{-n+1}$. In contrast, for every $\a\in\NN^d$ with $|\alpha| = n$,
 we have
$$
    \inf_{P \in \Pi_{n-1}^d}\| \x^\a - P\|_{[-1,1]^d} = 2^{-n+d},
$$
as shown in \rev{\cite{Yudin} (see also \cite{Sloss})}, which depends on the dimension $d$.
\end{exam}

  Since an optimal solution of problem \eqref{def-intro-multi-2} is given in terms of the univariate Chebyshev
    polynomial $T_n$ in Theorem  \ref{thm:supported}, the reader may wonder if this optimal solution is in fact unique. This turns out to be not the case, as the next example shows.
\begin{exam}\label{ex:simplex1}
The simplex
$\triangle^d =\{\,\x \in \RR^d \; : \; 0 \le x_1 \le x_2 \le \ldots \le x_d \le 1\,\}$ is diagonally-determined with $\v = (1,0,\ldots,0)$, and
$\mathrm{diag}(\triangle^d)  = [0,1]$. Hence, by Theorem \ref{thm:supported},
\begin{equation*} 
 \|P^*\|_{\triangle^d} = \inf \left  \{ \|P\|_{\triangle^d}: P \in \Pi_n^* \right \} = \left(\frac{1}{2}\right)^n \frac{1}{2^{n-1}} = 2^{-2n+1}
\end{equation*}
is independent of the dimension $d$. Moreover, one has $P^*(\x) = 2^{-2n+1}T_n(2x_1-1)$, which is univariate.
Another valid choice for $\v$ is $\v = (0,1,0,\ldots,0)$ that leads to a different least Chebyshev polynomial,
namely $\x \mapsto 2^{-2n+1}T_n(2x_2-1)$.
\end{exam}

\rev{
\begin{exam}\label{ex:simplex1}
Returning to Example \ref{ex:wol}, where $\Omega$ is the unit ball for the {\sf owl} norm,
we have $[a,b] = [-1/W,1/W]$ and $\v = \w/W$, so that, by Theorem \ref{thm:supported}, the least Chebyshev  polynomial of degree $n$ becomes
\[
  \x\mapsto P^*(\x) = \left(\frac{1}{W}\right)^n \frac{1}{2^{n-1}}   T_n\left(\langle \w,\x\rangle\right).
\]
\end{exam}
}



We point out that the requirement $\sum_{\a \in \NN^d_n} a_\a =1$, is susceptible to an affine change of variables other than a translation.
For instance, the triangle $\hat \triangle^2 = \{\,\x: x_1\ge 0, x_2 \ge 0, x_1 + x_2 \le 1\,\}$, is a mirror image of $\triangle^2$ under $x_2 \mapsto 1-x_2$.
However, under this affine change of variables, the leading monomial $a_0 x_1^2 + a_1 x_1 x_2 + a_2 x_2^2$ of a polynomial in $\Pi_2^*$ becomes,
$a_0 x_1^2 - a_1 x_1 x_2 + a_2 x_2^2$, which is, in general, no longer an element of $\Pi_2^*$.

\rev{
Moreover, the extremal problem \eqref{def-intro-multi-2} in Definition \ref{def:extremal_problem} is not rotationally invariant.
To see this consider the following three sets in $\mathbb{R}^2$, that are rotations of the same line segment:
\begin{eqnarray*}
\Omega_1 &:=& \{(x,x)  \; | \;  x \in [-1,1]\} \\
\Omega_2 &:=& \{(0,x)  \; | \;  x \in [-\sqrt{2},\sqrt{2}]\} \\
\Omega_3 &:=& \{(x,-x)  \; | \;  x \in [-1,1]\}. \\
\end{eqnarray*}
Of these three sets, only $\Omega_1$ is diagonally-determined.
Even though the three sets are rotations of each other, the solution to the extremal problem \eqref{def-intro-multi-2} is starkly different in each case:
\begin{itemize}
  \item For $\Omega_1$, Theorem \ref{thm:supported} provides the solution $(x_1,x_2) \mapsto \frac{1}{2^{n-1}}T_n(x_1)$ as a least Chebyshev polynomial, and thus $\inf_{P\in\Pi^*_n}\, \|P\|_{\Omega_1} = \frac{1}{2^{n-1}}$;
  \item For $\Omega_2$, a least Chebyshev polynomial is clearly given by $(x_1,x_2) \mapsto x_1^n$, so that $\inf_{P\in\Pi^*_n}\, \|P\|_{\Omega_2} = 0$;
  \item For $\Omega_3$, a least Chebyshev polynomial is
 $
  (x_1,x_2) \mapsto \frac{1}{2^n}(x_1+x_2)^n,
 $
   so that $\inf_{P\in\Pi^*_n}\, \|P\|_{\Omega_3} = 0$.
\end{itemize}
The example shows that the diagonal of $\Omega$ plays a key role in determining the least Chebyshev polynomial,
 and that it is not only an artefact of the analysis.}

\section{The dual framework and signatures}\label{sec:dual-framework-signatures}

In this section we consider the dual problem of~\eqref{def-intro-multi-2}.
Our goal is to show that, in the special case of diagonally-determined sets, the dual problem has a closed-form solution. Moreover, we will construct a dual solution that is atomic (discrete), and supported on $n+1$ points.

\subsection{The dual problem}
With $\Omega$ being a set in $\RR^d$, let $C(\Omega)$ be the space of continuous functions
on $\Omega$, and let $C(\Omega)^*$ be the dual space of $C(\Omega)$.
For compact $\Omega$, one has the following, strong duality result for problem \eqref{def-intro-multi-2}.
\begin{thm}
\label{thm:strong_duality_1_3}
Let $\Omega \subset \RR^d$ be compact, and consider the dual problem of \eqref{def-intro-multi-2}, namely
\begin{equation}
\label{eq:gamma*}
\gamma^* \coloneqq \sup_{\substack{\gamma \in \mathbb{R}\\ L \in C(\Omega)^*}}
\bigg\{ \gamma : \;  L_{| \Pi_{n-1}^d} = 0 , \, \|L\|_{C(\Omega)^*} = 1, \mbox{and }  L(\mathbf{x} \mapsto \mathbf{x}^\alpha) = \gamma \;\text{ for all } \alpha \in \NN^d_n \bigg\},
\end{equation}
where the norm of $L$ is defined as
\[
\|L\|_{C(\Omega)^*} = \sup_{h \in C(\Omega)} \{ |Lh| \; : \|h\|_{\Omega} \le 1\}.
\]
One has
$
\gamma^* \le \inf \left\{ \|P\|_\Omega \; : \; P\in\Pi^*_n\right\}
$
(weak duality), and
$
\gamma^* = \inf \left\{ \|P\|_\Omega \; : \; P\in\Pi^*_n\right\}
$
if $\Omega$ is compact (strong duality).
\end{thm}

One may derive this result directly through conic linear programming duality theory,  e.g.\ \cite[Proposition 2.9]{Shapiro2001}, and we will omit the proof here.

 It is insightful, though, to make a link with classical Chebyshev approximation. To this end, we denote an optimal
solution of problem \eqref{def-intro-multi-2} by
\begin{equation}
\label{eq:P*}
    \x\mapsto P^* (\x) = \sum_{\a \in \NN^d_n} a^*_\a \x^{\a} + Q^*(\x) \quad\hbox{such that} \quad  \sum_{\a \in \NN^d_n} a^*_{\a} = 1, \qquad Q^* \in \Pi^d_{n-1}.
\end{equation}
Note that $P^* \in \Pi_{n}^*$.
Clearly,
\[
\|P^*\|_\Omega = \min_{Q \in \Pi_{n-1}^d} \left\|\sum_{\a \in \NN^d_n} a^*_\a \x^{\a} + Q(\x) \right\|_\Omega = \min_{Q \in \Pi_{n-1}^d} \left\|\sum_{\a \in \NN^d_n} a^*_\a \x^{\a} - Q(\x) \right\|_\Omega,
\]
while the latter problem is the classical Chebyshev problem of approximating the homogeneous polynomial $\x \mapsto \sum_{\a \in \NN^d_n} a^*_\a \x^{\a}$  from $\Pi_{n-1}^d$.

Recall the strong duality result for the classical Chebyshev approximation problem as given in the paper by Rivlin and Shapiro \cite[Corollary 2]{RS}.\footnote{In \cite[Corollary 2]{RS}, the objective is in fact
given as $\sup |L(f)|$, but the absolute value may be omitted without loss of generality, since $L$ is feasible for the dual problem if and only if $-L$ is feasible.}

\begin{thm}
For any $f \in C(\Omega)$, with $\Omega$ compact, one has the following
\begin{equation}
\label{eq:Chebyshev_strong_dual}
 \min_{Q \in \Pi_{n-1}^d} \left\|f - Q \right\|_\Omega
= \max_{ L \in C(\Omega)^*}
\bigg\{ L(f) : \;  L_{| \Pi_{n-1}^d} = 0 , \, \|L\|_{C(\Omega)^*} = 1\bigg\}.
\end{equation}
\end{thm}

From this, we can immediately deduce the following relation between the two dual problems.

\begin{lem}
\label{lemma:equiv_duals}
For $f(\x) = \sum_{\a \in \NN^d_n} a^*_\a \x^{\a}$ with $a^*_\a$ as defined in \eqref{eq:P*}, and $\Omega$ compact,
 any optimal solution of problem  \eqref{eq:gamma*} is also optimal for the dual problem in \eqref{eq:Chebyshev_strong_dual}.
\end{lem}
\begin{proof}
    Let $L^*$ denote an optimal solution of problem \eqref{eq:gamma*}.
    Then $L^*$ is feasible for the dual (maximization) problem in \eqref{eq:Chebyshev_strong_dual}. Using that $\sum_{\a \in \NN^d_n} a^*_{\a} = 1$, one has
    \[
L^*(f) = L^*\left(\sum_{\a \in \NN^d_n} a^*_\a \x^{\a} \right) = \sum_{\a \in \NN^d_n} a^*_\a L^*(\x^{\a}) = \sum_{\a \in \NN^d_n} a^*_\a \gamma^* = \gamma^* = \|P^*\|_\Omega,
    \]
    where the last equality is the strong duality relation in Theorem \ref{thm:strong_duality_1_3}.
    Since the optimal value of the primal problem \eqref{eq:Chebyshev_strong_dual} is also $\|P^*\|_\Omega$ when $f(\x) = \sum_{\a \in \NN^d_n} a^*_\a \x^{\a}$,
        it follows that $L^*$ is an optimal solution for the dual problem in \eqref{eq:Chebyshev_strong_dual}, as required.
\end{proof}
In the next subsection, we review the fact that the optimal dual solutions may be assumed to be atomic (discrete) without loss of generality.

\subsection{Atomic dual solutions and signatures}

We first recall a classical interpolation formula for linear functionals, as given in  \cite[Corollary 3]{RS}.
\begin{thm}
\label{thm:atomic_linear_functionals}
    Let $L$ be any linear functional on a finite dimensional subspace $V$ of $C(\Omega)$, with $\Omega$ compact.
    Then there exist points $\omega_1,\ldots,\omega_r \in \Omega$ with $r \le \mathrm{dim}(V)$, and non-zero
    scalars $\tau_1,\ldots,\tau_r$, such that, defining the point evaluation functionals
    \[
L_{\omega_i}(f) = f(\omega_i) \quad i = 1,\ldots,r,
    \]
    one has
    \begin{equation}
        L = \sum_{i=1}^r \tau_i L_{\omega_i}, \quad \|L\|_{C(\Omega)^*} = \sum_{i=1}^r |\tau_i|.
        \label{eq:atomic_sol_dual}
    \end{equation}
\end{thm}
As a consequence, there exists a discrete (atomic) solution to problem~\eqref{thm:strong_duality_1_3}.
We give a proof below only for the sake of completeness and later reference --- the type of argument we use is classical.
\begin{cor}\label{cor:extremal-points}
 There exist points $\omega_1,\ldots,\omega_r \in \Omega$ with $r \le \mathrm{dim}(\Pi_{n}^d)$, and non-zero scalars $\tau_1,\ldots,\tau_r$ with $\sum_{i=1}^r |\tau_i| = 1$, such that an optimal solution of problem~\eqref{thm:strong_duality_1_3} is given by~\eqref{eq:atomic_sol_dual}. Moreover, the points
   $\omega_1,\ldots,\omega_r$ are extremal points of any optimal solution $P^*$ to problem \eqref{def-intro-multi-2}, i.e.,
    \[
|P^*(\omega_i)| = \|P^*\|_\Omega   \quad \text{ for all } i = 1,\ldots,r.
    \]
\end{cor}
\begin{proof}

    Let $L^*$ denote an optimal solution of problem
\eqref{thm:strong_duality_1_3}.
The first statement in the corollary now follows immediately from Theorem~\ref{thm:atomic_linear_functionals}, applied to $L^*$, and using $V = \Pi_{n}^d$, so that
\begin{equation}
\label{eq:L*}
    L^* = \sum_{i=1}^r \tau_i L_{\omega_i}, \quad  \sum_{i=1}^r |\tau_i| = 1.
\end{equation}

 It remains to show that the points  $\omega_1,\ldots,\omega_r$ are extremal points of any optimal
    solution $P^*$ to problem \eqref{def-intro-multi-2}.

By Lemma~\ref{lemma:equiv_duals}, $L^*$ is also an optimal solution for the dual problem in~\eqref{eq:Chebyshev_strong_dual}.
Thus, for any optimal solution $P^*$ of problem~\eqref{def-intro-multi-2}, one has
\[
\|P^*\|_\Omega = L^*(P^*) = \sum_{i=1}^r \tau_i P^*(\omega_i).
\]
Since $\sum_{i=1}^r |\tau_i| = 1$, $\|P^*\|_\Omega$ is a weighted average of the values $P^*(\omega_i)$ with $\tau_i > 0$ and $-P^*(\omega_i)$ with $\tau_i < 0$.
As a consequence, we have
\[
\|P^*\|_\Omega = \left\{
\begin{array}{cc}
\phantom{-}P^*(\omega_i) & \mbox{ if $\tau_i > 0$,} \\
-P^*(\omega_i) & \mbox{if $\tau_i <0$,}
\end{array}
\right.
\]
completing the proof.
\end{proof}

The extremal points in Corollary~\ref{cor:extremal-points} are usually called the \textit{support of an extremal signature}, defined along the line of \cite[Section 2.2]{Riv}
as follows.

\begin{defn}
\label{def:signature}
A \textit{signature} with finite support $S \subset \Omega$ is simply a (partition) function from $S$ to $\{\pm 1\}$.
A signature $\sigma$ with support $S$ is said to be \textit{extremal} for a subspace $V \subset C(\Omega)$ if there exist weights $\lambda_\omega > 0$, $\omega \in S$, such that $\sum_{\omega \in S} \lambda_\omega \sigma(\omega) v(\omega) = 0$ for all $v \in V$.
A signature $\sigma$ with support $S$ is said to be \textit{associated} with a function $g \in C(\Omega)$ if $S$ is included in the set  $\{\omega \in \Omega: |g(\omega)| = \|g\|_\Omega \}$ of extremal points of $g$  and if $\sigma(\omega) = \sgn(g(\omega))$ for all $\omega \in S$.
\end{defn}

Thus, the result of Corollary~\ref{cor:extremal-points} may be restated as the existence of a signature with support $S = \{\omega_1,\ldots,\omega_r \} $ where $r \le \mbox{dim}(\Pi_{n}^d)$, so that this signature is extremal for $\Pi_{n-1}^d$, and associated with every optimal solution of problem \eqref{def-intro-multi-2}. Therefore, we will simply refer to an {\it optimal signature} for problem \eqref{def-intro-multi-2}. Formally we have the following result.

\begin{prop}
\label{prop:signatures}
    Any atomic solution of the dual problem \eqref{thm:strong_duality_1_3}, say $L^*$ of the form \eqref{eq:L*}, gives rise to an optimal signature for problem
    \eqref{def-intro-multi-2}, by setting:
    \[
    S = \{\omega_1,\ldots,\omega_r\}, \quad \sigma(\omega_i) = \sgn(\tau_i), \quad \lambda_i = |\tau_i|, \qquad i = 1,\ldots,r.
    \]
\end{prop}

This is essentially a reformulation of the well-known characterization of best approximation by polynomials in terms of the extremal signature (\cite[Theorem 2.6]{Riv}) for the dual problem.

\subsection{Signatures for diagonally-determined sets}
For the special case where $\Omega$ is a diagonally-determined set, we may now infer information about an optimal signature from Theorem \ref{thm:supported} and Proposition \ref{prop:signatures}.
The idea of the proof is to construct an atomic optimal solution for the dual problem \eqref{thm:strong_duality_1_3}.

\begin{thm}
\label{thm:sinature_diag_supp}
    Let $\Omega$ be a diagonally-determined set with $\mathrm{diag}(\Omega) = [a,b]$ and vector $\v$. Then, an optimal signature for problem~\eqref{def-intro-multi-2} is defined by a set $S$, of $n+1$ points
    \begin{equation}    \label{eq:tau_diag_supp}
\omega_j = a\mathbf{1} + \frac{b-a}{2}\left(1 + \cos\left(\frac{(j-1)\pi}{n} \right) \right){\mathbf{1}} \quad \text{ for } j= 1,\ldots,n+1,     \end{equation}
where ${\mathbf{1}}$  again  denotes the all-ones vector in $\mathbb{R}^d$, together with the
partition function $\sigma:S \rightarrow \{-1,1\}$, given by
\[
\sigma(\omega_j) =
\left\{
\begin{array}{cc}
\phantom{-}1 & \mbox{if $j$ is odd} \\
-1  & \mbox{if $j$ is even,}
\end{array}
\right.
\]
as well as $\lambda_1 = \lambda_{n+1} = \frac{1}{2}$, and $\lambda_i = 1$ if $i=2,\ldots, n$.
\end{thm}
\begin{proof}
 By Theorem \ref{thm:supported}, an optimal solution to problem
 \eqref{def-intro-multi-2} is given by the polynomial $P^*$ in \eqref{eq:minimizer}, namely
 \[
 \x \mapsto P^*(\x) = \left(\frac{b-a}{2}\right)^n \frac{1}{2^{n-1}}   T_n\left(- 1+ 2 \,\frac{\langle \v,\x\rangle -a}{b-a}\right).
 \]
 Using the well-known fact that the extremal points of $T_n$ are given by the Gauss-Lobatto-Chebyshev points $\xi_j \coloneqq \cos\left((j-1)\pi/n\right)$ for $j=1,\ldots, n+1$,
  we have that the points listed in \eqref{eq:tau_diag_supp} are extremal points of $P^*$. Indeed, using $\langle \v,\mathbf{1}\rangle = 1$, one has
\[
\langle \v,\omega_j\rangle  = a + \frac{b-a}{2}\left(1 + \cos\left(\frac{(j-1)\pi}{n} \right) \right) \quad \text{ for } j= 1,\ldots,n+1,
\]
which is the same as
\[
- 1+ 2 \,\frac{\langle \v,\omega_j\rangle -a}{b-a} = \xi_j  \text{ for } j=1,\ldots, n+1.
\]

 Next,  we will construct an atomic solution of the dual problem~\eqref{thm:strong_duality_1_3}, which will lead to the required optimal signature, by Proposition~\ref{prop:signatures}.

 To this end, define the linear operator
 \[
 \rev{
    L^* = \frac{1}{n}\sum_{i=1}^{n+1} (-1)^{i+1}\lambda_iL_{\omega_i},
    }
 \]
 so that $\|L^*\|_{C(\Omega)^*} = \frac{1}{n}\sum_{i=1}^{n+1} |(-1)^{i+1}\lambda_i| = 1$, by construction.

 We claim that $L^*$ is an optimal solution of the dual problem \eqref{thm:strong_duality_1_3}. We first verify feasibility.
  We will show that $L^*$ vanishes on $\Pi^d_{n-1}$, as required, by using the Gauss-Lobatto-Chebyshev quadrature formula (see e.g.\ \cite[Exercise 1.5.29]{Riv}), which implies, for $k<n$:
\rev{
 \[
0= \int_{-1}^1 T_n(t)T_k(t)\frac{1}{\sqrt{1-t^2}}dt = \frac{\pi}{n}\sum_{j=1}^{n+1} \lambda_j T_n(\xi_j)T_k(\xi_j),
 \]
 }
 due to the orthogonality of Chebyshev polynomials. Using $T_n(\xi_j) = (-1)^{1+j}$ yields
 \rev{
\begin{equation}
\label{eq:orthogonal}
L^*(T_k) = \frac{1}{n}\sum_{j=1}^{n+1} \lambda_j (-1)^{1+j}T_k(\xi_j) = 0 \quad \mbox{ if $k<n$.}
\end{equation}
}
Thus $L^*$ vanishes on
univariate (and separable) polynomials in $\Pi^d_{n-1}$. For the general case, it suffices to consider  $d=2$ and the monomial $p(x_1,x_2) = T_{k_1}(x_1)T_{k_2}(x_2)$, where $k_1+k_2 < n$, and show that $L^*(p)  = 0$. To this end, note that
\rev{
\begin{eqnarray*}
L^*(p) &=& \frac{1}{n}\sum_{j=1}^{n+1} \lambda_j (-1)^{1+j}T_{k_1}(\xi_j)T_{k_2}(\xi_j) \\
       &=& \frac{1}{n}\sum_{j=1}^{n+1} \lambda_j (-1)^{1+j} \frac{1}{2}\left(T_{|k_1-k_2|}(\xi_j)+T_{k_1+k_2}(\xi_j)\right) \\
        &=& \frac{1}{2n}\sum_{j=1}^{n+1} \lambda_j (-1)^{1+j} T_{|k_1-k_2|}(\xi_j)+ \frac{1}{2n}\sum_{j=1}^{n+1} \lambda_j (-1)^{1+j}T_{k_1+k_2}(\xi_j) = 0,\\
\end{eqnarray*}
}
using \eqref{eq:orthogonal}. The result for general $d$ now follows by induction.

Secondly, we verify that $L^*(\mathbf{x}^\alpha)$ is independent of $\alpha$, when $|\alpha| = n$, say  $L^*(\mathbf{x}^\alpha) = \hat \gamma$ whenever $\alpha \in \NN^d_n$. This follows immediately from the construction of $L^*$, since all coordinates of the vectors $\omega_i$ are equal for all $i=1,\ldots, n+1$.
Finally, to prove optimality, we note that, for $P^*$ written in the form \eqref{eq:P*},
\[
\|P^*\|_\Omega = L^*(P^*) = \sum_{\a \in \NN^d_n}a^*_\alpha L^*(\mathbf{x}^\alpha) = \hat \gamma\sum_{\a \in \NN^d_n}a^*_\alpha = \hat \gamma,
\]
so, by the weak duality theorem,   $L^*$ is indeed optimal. The required result now follows from Proposition~\ref{prop:signatures}.
\end{proof}

\begin{rem}
    It is interesting to note that --- for diagonally-determined sets --- there exists an optimal signature for problem~\eqref{def-intro-multi-2} with support size $n+1$. One can compare this to the general result for compact $\Omega$ of a signature support size of $\mbox{dim}(\Pi^d_n) = {n+d \choose d}$.
    Remarkably the minimum support size is independent of $d$  in the diagonally-determined case. Moreover, we did not use compactness of $\Omega$ in the proof of Theorem~\ref{thm:sinature_diag_supp},
     whereas compactness is needed for the general bound. Finally, note that the optimal signature in Theorem \ref{thm:sinature_diag_supp} depends on $\Omega$ only through its diagonal $[a,b]$, and is independent of $\v$.
\end{rem}

\section{Detecting the diagonally-determined property}
\label{sec:detection}
The goal of this section is to provide  a numerical scheme to detect whether a given compact basic semi-algebraic set
\begin{equation}
\label{set-omega}
\Omega\,:=\,\{\,\x\in \mathbb{R}^d: g_j(\x)\geq0\,,\quad j=1,\ldots,m\,\}
\end{equation}
for some given polynomials $g_j$, is diagonally-determined.

The {\em quadratic module} of  $g_1,\ldots,g_m$, denoted by $\mathcal{Q}(g_1,\ldots,g_m)$, is defined as the set of all polynomials of the form
\begin{equation}
\label{eq:quadratic module}
p\,=\,\sigma_0+\sigma_1\,g_1+\cdots +\sigma_m\,g_m\,
\end{equation}
for some sum-of-squares polynomials $\sigma_0,\sigma_1,\ldots,\sigma_m$. Clearly, if $p \in \mathcal{Q}(g_1,\ldots,g_m)$, then $p(\x) \ge 0$ for all $\x \in \Omega$.

\begin{assumption}
\label{assumption:Archimedean}
We  assume that $\mathcal{Q}(g_1,\ldots,g_m)$ satisfies  the Archimedean condition,
meaning  that
there exists   $p \in \mathcal{Q}(g_1,\ldots,g_m)$ such that
$$
\left\{ \x \in \mathbb{R}^d \; : \; p(\x) \ge 0\right\} \mbox{ is compact.}
$$
\end{assumption}

Under this assumption, if a polynomial is (strictly) positive on $\Omega$, then it belongs to $\mathcal{Q}(g_1,\ldots,g_m)$.
We state this result --- known as Putinar's Positivstellensatz --- as a theorem for later use.

\begin{thm}[Putinar \cite{Put}]
\label{thm:Putinar}
Assume $\Omega$ is of the form \eqref{set-omega} and that Assumption \ref{assumption:Archimedean} holds.
If $p \in \Pi^d_n$ satisfies $p(\x) > 0$ for all $\x \in \Omega$, then $p \in \mathcal{Q}(g_1,\ldots,g_m)$.
\end{thm}

We also introduce the concept of the {\em truncated quadratic module} of degree $k \in \mathbb{N}_0$, denoted by $\mathcal{Q}(g_1,\ldots,g_m)_k$,
as the subset of $\mathcal{Q}(g_1,\ldots,g_m)$ where each term on the right-hand-side of \eqref{eq:quadratic module} has degree at most $k$, i.e.\
deg$(\sigma_0) \le k$, and deg$(\sigma_jg_j) \le k$ for all $j = 1,\ldots,m$. The membership problem for $\mathcal{Q}(g_1,\ldots,g_m)_k$ may
 be reformulated as a semidefinite programming feasibility problem; see Lasserre \cite{lass-siopt-01}, and the references therein.

\subsection*{Finding the diagonal $[a,b]$}
We first note that one may find the diagonal of $\Omega$ by finding the real roots of the univariate polynomials
\[
t \mapsto p_j(t) := g_j(t \mathbf{1}), \quad j=1,\ldots,m.
\]
Indeed, for each $p_j$ we may find the set of all (possibly unbounded) intervals where it is nonnegative,
and subsequently take the intersection of all these sets for $j=1,\ldots,m$.
If the result is a single interval $[a,b]$, then $\mathrm{diag}(\Omega) = [a,b]$, and we may proceed with the next step, since the first property holds in Definition \ref{def:diag_concentration}.

\subsection*{Once $[a,b]$ is known}
We assume that we now know $[a,b]=\mathrm{diag}(\Omega)$, and that $a < 0 < b$. The latter assumption is without loss of generality, due to Lemma \ref{lemma:translation}.
Then, to detect whether the second property holds in Definition \ref{def:diag_concentration}, consider the optimization problem:
\begin{equation}
\label{detection-v}
\rho=\displaystyle\min_{\v'} \,\{\,|1-\langle\v',\mathbf{1}\rangle|:\:
a\leq\langle\v',\x\rangle\,\leq\,b\quad\,\forall \x\in\Omega\,\}\,.
\end{equation}
Then $\Omega$ is diagonally-determined if and only  if $\rho = 0$. Note that one may in fact omit the absolute value in the objective of \eqref{detection-v},
since $\x = b\mathbf{1} \in \Omega$, and $\langle\v',\x\rangle \leq b$ therefore implies $\langle\v',\mathbf{1}\rangle \leq 1$ for all $\v'$ that are feasible for \eqref{detection-v}.

We now relax this problem to a sequence of semidefinite programs, using the approach introduced by Lasserre \cite{lass-siopt-01}.
The key idea is to replace the nonnegativity conditions in \eqref{detection-v} by the sufficient condition of membership of a truncated quadratic module.
To this end, we introduce the two affine functions, for given $\v' \in \mathbb{R}^d$:
\[
\x \mapsto b-\langle\v',\x\rangle = \ell_{\v',b}(\x), \; \x \mapsto \langle\v',\x\rangle-a = \ell_{\v',a}(\x),
\]
to obtain the reformulation of problem \eqref{detection-v} as
\[
\rho=\min_{\v'} \,\{\,1-\langle\v',\mathbf{1}\rangle:\:
\ell_{\v',a}(\x),\ell_{\v',b}(\x) \ge 0 \quad\,\forall \x\in\Omega\,\}\,.
\]
For every $k\in \mathbb{N}_0$, we now define:
\begin{equation}
\label{detection-v-relax}
\rho_k =  \min_{\v'}\left\{  1-\langle\v',\mathbf{1}\rangle \; : \;  \ell_{\v',a},\ell_{\v',b} \in \mathcal{Q}(g_1,\ldots,g_m)_k\right\}.
\end{equation}
Importantly, and as already mentioned, we emphasize that the optimization problem \eqref{detection-v-relax} is a
 (convex) semidefinite program that can be solved efficiently with arbitrary precision (fixed in advance),
 under mild assumptions, using, e.g., interior point methods; see Lasserre \cite{lass-siopt-01}.

\begin{prop}
Assume $\Omega$ is of the form \eqref{set-omega}, that Assumption \ref{assumption:Archimedean} holds, and that
$\mathrm{diag}(\Omega) = [a,b]$ where $a < 0 < b$.
Then one has $\rho_k \in [0,1]$ for all $k$. Moreover, if  $\Omega$ is diagonally-determined,
then the non-increasing sequence $\{\rho_k\}$ converges to zero.
If $\rho_k = 0$ for some finite value of $k$, then
the optimal $\v'\in \mathbb{R}^d$ of \eqref{detection-v-relax} certifies that $\Omega$ is diagonally-determined with vector $\v = \v'$ and diagonal $[a,b]$.
\end{prop}
\begin{proof}
For  $k = 0$, a feasible solution of \eqref{detection-v-relax} is given by $\v' = \mathbf{0}$, since $a < 0 < b$,
so that $\rho_0 \le 1$. Consequently,  $\rho_k \in [0,1]$ for all $k$, since $\mathcal{Q}(g_1,\ldots,g_m)_k \subseteq \mathcal{Q}(g_1,\ldots,g_m)_{k+1}$ for all $k \in \mathbb{N}_0$.
Assume now that $\Omega$ is diagonally-determined with vector $\v$ and diagonal $[a,b]$.
Then, for any $\varepsilon > 0$ and $\x \in \Omega$, one has
\[
\langle (1-\varepsilon)\v,\x\rangle \in [(1-\varepsilon)a,(1-\varepsilon)b] \subset (a,b),
\]
where we again use that $a < 0 < b$.
Thus,
$\x \mapsto \langle (1-\varepsilon)\v,\x\rangle -a \in \mathcal{Q}(g_1,\ldots,g_m)$ and
$\x \mapsto b-\langle (1-\varepsilon)\v,\x\rangle \in \mathcal{Q}(g_1,\ldots,g_m)$, by Theorem \ref{thm:Putinar}.
This implies that, for sufficiently large $k^* \in \mathbb{N}_0$, $\v' = (1-\varepsilon)\v$ is feasible for \eqref{detection-v-relax} with $k = k^*$.
Thus $\rho_{k^*} \le 1- \langle(1-\varepsilon)\v,\mathbf{1}\rangle = \varepsilon$,
and therefore the sequence $\{\rho_k\}$ is non-increasing, and converges  to zero.
\end{proof}

\begin{exam}
It  is insightful to see how the above procedure works in the case of the unit ball in the $\infty$-norm, i.e.\
\[
\Omega = \{\x \in \mathbb{R}^d \; : \; \|\x\|_\infty \le 1\}.
\]
By Proposition \ref{prop:ball}, and its proof, we know that $\Omega$ is diagonally-determined, with diagonal $[a,b] = [-1,1]$ and vector $\v = \frac{1}{d}\mathbf{1}$.
We may describe $\Omega$ as a semi-algebraic set in the form \eqref{set-omega} by using the polynomials $g_j(\x) = 1 - x_j^2$ for $j=1,\ldots,d$, i.e.\ $m = d$.
The next step is to find the diagonal of $\Omega$ by finding the roots of the univariate polynomials
$p_j(t) := g_j(t \mathbf{1})= 1-t^2$ for $j=1,\ldots,d$. These polynomials are the same for all $j$ and nonnegative on the interval $[-1,1]$, so that we
obtain $[a,b] = [-1,1]$.
The next step is to solve \eqref{detection-v-relax} for increasing values of $k\in \mathbb{N}$.
Using the identity
\[
 1 \pm x  = \frac{1}{2} \left( (1 \pm x)^2 + 1 - x^2 \right), \quad x \in \mathbb{R},
\]
one obtains that
\[
b - \langle \v,\x\rangle = 1 - \frac{1}{d}\sum_{j=1}^d x_j = \frac{1}{d}\sum_{j=1}^d\left(1- x_j\right) = \frac{1}{2d}\sum_{j=1}^d \left((1 - x_j)^2 + g_j(\x) \right).
\]
Similarly one has
\[
\langle \v,\x\rangle - a = \frac{1}{2d}\sum_{j=1}^d \left((1 + x_j)^2 + g_j(\x) \right) \in \mathcal{Q}(g_1,\ldots,g_d)_2.
\]
Thus, for $k=2$, we find that $\v'= \v$ is feasible in \eqref{detection-v-relax}, so that $\rho_2 = 0$.
\end{exam}
In the example we have finite convergence of the sequence $\{\rho_k\}$, but it is an open question when this occurs.
There are known sufficient conditions for the finite convergence of Lasserre-type semidefinite programming hierarchies, e.g.\ \cite{Nie},
but these are not satisfied for our problem in general.

\section{Conclusion and discussion}
\label{sec:conclusions}
\rev{
The extremal problem \eqref{def-intro-multi-2} we considered in this paper relies on a specific generalization of the concept of a monic univariate polynomial to the multivariate case.
In particular, in the multivariate case we require that the leading coefficients (in the standard monomial basis) sum to one, leading to the introduction of the set $\Pi^*_n$ in Definition \ref{def:extremal_problem}. There are other natural choices, like also requiring the leading coefficients to be nonnegative.
Having said that, the least Chebyshev polynomial in Theorem \ref{thm:supported} already has
nonnegative leading coefficients for many diagonally-determined sets, e.g.\ the unit ball in a monotone norm $\|\cdot\|$.
To see this, let $\Omega$ be such a unit ball, centered at the origin.
By Proposition \ref{prop:ball}, $\Omega$ is then diagonally-determined with diagonal $[a,b] = [-1/\|\mathbf{1}\|,1/\|\mathbf{1}\|]$ and allows an entrywise-nonnegative vector $\v$.
By Theorem \ref{thm:supported}, the least Chebyshev polynomial in this case is given by
\[
 \x\mapsto P^*(\x) = \left(\frac{1}{\|\mathbf{1}\|}\right)^n \frac{1}{2^{n-1}}   T_n\left(\|\mathbf{1}\|\langle \v,\x\rangle\right).
\]
Since $\frac{1}{2^{n-1}}T_n$ is a monic univariate polynomial of degree $n$, the leading coefficients of the least Chebyshev polynomial $P^*$ are therefore
obtained from the multinomial expansion of $\x \mapsto \langle \v,\x\rangle^n$. Since $\v$ is entrywise-nonnegative by assumption, it follows that these leading coefficients are also nonnegative.
In this sense our generalization of monic polynomials to the multivariate case is a natural one, but is remains an interesting question which diagonally-determined sets allow a entrywise-nonnegative vector $\v$.
}

\rev{
A second point to mention, is that we have not explored the geometric and invariance properties of diagonally-determined
 sets to any great extent in this paper, and this deserves to be investigated further.
 The same may be said for the decision problem of whether a given set is diagonally-determined.
  If $\Omega$ is a compact semi-algebraic set, one may formulate an algorithmic procedure to decide this question, as we did in Section \ref{sec:detection}.
   An interesting special case is the (complexity) question: may one decide in polynomial time if a given polyhedron is diagonally-determined?
}

\rev{
Finally, we note that our main results  in Theorems \ref{thm:supported} and \ref{thm:sinature_diag_supp} are relatively straightforward to prove, once the correct formulations have been discovered.
We `discovered' these results by doing numerical experiments on least Chebyshev polynomials for the Euclidean unit ball, using the tools
described in our earlier paper \cite{caltex-1}. Thus we observed the structure of the atomic dual solution in Theorem \ref{thm:sinature_diag_supp}, and we could infer the main results from there.
In the immortal words of Bernhard Riemann:\footnote{As quoted in I.\ Lakatos, \emph{Proofs and Refutations} (Cambridge University Press, 1976).} \begin{quote}
{\em If only I had the theorems! Then I should find the proofs easily enough.}
\end{quote}
}

\subsection*{Acknowledgments}
 {This work is the result of a collaboration made possible by the SQuaRE program at the American Institute of Mathematics (AIM).
We are truly grateful to AIM for
the supportive and mathematically rich environment they provided.
In addition, we owe thanks to several funding agencies,
as
M. D. is supported by the Australian Research Council Discovery Early Career Award DE240100674,
S. F. is partially supported by grants from the National Science Foundation (DMS-2053172) and from the Office of Naval Research (N00014-20-1-2787),
E. de K. is supported by grants from the Dutch Research Council (NWO)
(OCENW.M.23.050 and OCENW.GROOT.2019.015),
J. B. L. is supported by the AI Interdisciplinary Institute  through the French program ``Investing for the Future PI3A" (ANR-19-PI3A-0004)
and by the National Research Foundation, Singapore,
through the DesCartes and Campus for Research Excellence and Technological Enterprise (CREATE) programs,
and Y. X. is partially supported by the Simons Foundation (grant \#849676).}

\end{document}